\documentclass[11pt]{article}

\usepackage{mathrsfs}
\usepackage{amsfonts}
\usepackage{titlesec}
\usepackage{enumerate}
\usepackage[toc,page,title,titletoc,header]{appendix}
\usepackage{amsmath,amsthm,amssymb,anysize}
\usepackage[numbers,sort&compress]{natbib}
\usepackage{color}
\usepackage{graphicx}
\usepackage{threeparttable}
\usepackage[all]{xy}

\theoremstyle{definition}
\newtheorem{thm}{Theorem}[subsection]
\newtheorem{lem}{Lemma}[section]

\newtheorem{prop}[lem]{Proposition}
\newtheorem{defn}[lem]{Definition}
\newtheorem{cor}[lem]{Corollary}

\newtheorem{exa}{Example}
\newtheorem{remark}{Remark}
\newtheorem{conclusion}{Conclusion}
\marginsize{3cm}{3cm}{0.5cm}{2.2cm}
\numberwithin{equation}{section}
\numberwithin{equation}{subsection}

\newcommand{\C}{\mathbb{C}}

\title{\textsf{On deformations of metric Lie superalgebras}}
\author{Yong Yang $^{a,b}$
\\
\\ \small\textit{$^{a}$School of Mathematics, Jilin University, Changchun 130012, China}
\\ \small\textit{$^{b}$Institute of Physics, University of P\'{e}cs, P\'{e}cs 7622, Hungary}
  }
\date{ }

\begin{document}
\makeatletter
\newcommand{\rmnum}[1]{\romannumeral #1}
\newcommand{\Rmnum}[1]{\expandafter\@slowromancap\romannumeral #1@}
\makeatother
\maketitle
\let\thefootnote\relax\footnote{
E-mail address: yangyong195888221@163.com}

\begin{quotation}
\small\noindent \textbf{Abstract}:
In this paper we study  metric deformations of indecomposable  metric Lie superalgebras with dimensions $\leq 6$. We consider formal deformations obtained by even cocycles, because the odd ones can not be used for constructing formal deformations.

\vspace{0.2cm} \noindent{\textbf{Keywords}}: (metric) Lie superalgebra; cohomology; (metric) deformation

\vspace{0.2cm} \noindent{\textbf{Mathematics Subject Classification 2020}}:  14D15, 17B56

\end{quotation}
\setcounter{section}{-1}
\section{Introduction}
Lie algebras endowed with an invariant bilinear form are important objects  in Lie theory and mathematical physics. A Lie algebra endowed with a non-degenerate, symmetric, invariant bilinear form is called a metric (or quadratic) Lie algebra. Examples of metric Lie algebras are semi-simple Lie algebras with an invariant bilinear form  given by the Killing form, due to the Cartan's criterion. However, there are still many metric solvable Lie algebras although their Killing forms are degenerate. Thus, it seems not easy to classify metric Lie algebras, even in low-dimensional case. Recently, deformations of low-dimensional  complex metric Lie algebras and their real forms of dimension $\leq 6$ was given by A. Fialowski and M. Penkava \cite{F-P}.

In the study of metric Lie algebras, an effective method to construct them is by double extension, which can be regarded as a combination of  central extension and  semi-direct product, introduced by V. G. Kac \cite{Kac-double} for solvable Lie algebras.
Metric Lie algebras  were described inductively, based on double extensions, by A. Medina, P. Revoy in indecomposable, non-simple case \cite{M-R} and by G. Favre, L. Santharoubane in non-trivial center case \cite{F-S}. Another interesting method is $T^{\ast}$-extension which can be regarded as a semi-direct product of a Lie algebra and its dual space by means of the coadjoint representation, introduced by M. Bordemann \cite{MB}.
M. Bordemann proved that every finite-dimensional nilpotent metric Lie algebra of even
dimension can by obtained by a $T^{\ast}$-extension.
However, $T^{\ast}$-extension does not exhaust all possibilities for constructing metric Lie algebras of even dimension. In fact, for the case of dimension $\leq 6$, there is only one class of non-Abelian metric Lie algebras which can be obtained by a $T^{\ast}$-extension \cite{F-P}.

H. Benamor and S. Benayadi generalized the notion of double extension to metric Lie superalgebras by considering supersymmetric invariant bilinear form and proved that every non-simple indecomposable metric Lie superalgebra with 2-dimensional odd part is a double extension of a one-dimensional or
semi-simple Lie algebra \cite{B-B}. It also holds for the metric Lie superalgebra with 2-dimensional even part \cite{2}. In the past years  the study of metric Lie superalgebras became intensive. Many classes of metric Lie superalgebras have been studied \cite{2,B,A-B-B}. Different from what happens in the Lie case, the Killing form is not always non-degenerate on a semi-simple Lie superalgebra.  So it becomes more difficult to
classify metric Lie superalgebras, even for semi-simple cases.
An interesting fact that any indecomposable non-simple metric Lie superalgebra of dimension $\leq 6$ is a double extension of a 1-dimensional Lie algebra. A classification of indecomposable metric Lie superalgebras with dimension $\leq 6$ was obtained in \cite{class}.
However, there are no results for  deformations of metric Lie superalgebras even in low-dimensional cases.
For a metric Lie superalgebra, an interesting question is which deformations are  metric.

The aim of this paper is to study  metric deformations  of indecomposable metric Lie superalgebras with dimensions $\leq 6$. In Section 1, we recall the basic definitions and results for metric Lie superalgebras. In Section 2, we recall the cohomology and deformation theory for Lie superalgebras. In Section 3, we compute metric deformations of  indecomposable metric Lie superalgebras with dimensions $\leq 6$ and show which ones are metric among these deformations. Throughout the paper, the ground field is supposed to be the complex field $\C$.

\section{Metric Lie superalgebra}
In this section, we recall the basic facts for metric Lie superalgebras, introduced in \cite{B-B,2}.
We also introduce the notion of double extension and give some examples.

\subsection{Structure of metric Lie superalgebras}
Recall that a Lie superalgebra $\mathfrak{g}=\mathfrak{g}_{\bar{0}}\bigoplus \mathfrak{g}_{\bar{1}}$ is a $\mathbb{Z}_{2}$-graded algebra whose multiplication $[\ , \ ]$ satisfies the skew-supersymmery and super Jacobi identity, i.e.
$$(-1)^{|x||z|}[x,[y,z]]+(-1)^{|y||x|}[y,[z,x]]+(-1)^{|z||y|}[z,[x,y]]=0,$$
for any $x$, $y$, $z$ in $\mathfrak{g}$ \cite{Musson,Kac}. For a homogeneous element $x\in \mathfrak{g}_{\bar{0}}\bigcup \mathfrak{g}_{\bar{1}}$, write $|x|$ for the $\mathbb{Z}_{2}$-degree of $x$. Call an ideal of a Lie superalgebra a $\mathbb{Z}_{2}$-graded ideal.
 Call a homomorphism between superalgebras the one that preserves $\mathbb{Z}_{2}$-grading. The definition of solvable Lie superalgebras is the same as for Lie algebras. A Lie superalgebra $\mathfrak{g}$ is called \emph{simple (semi-simple)} if $\mathfrak{g}$ is not Abelian and does not contain nontrivial (solvable) ideals. Recall that a bilinear form $B$ on a Lie superalgebra $\mathfrak{g}$ is called \emph{invariant} if
$$B([x,y],z)=B(x,[y,z])$$
for any $x$, $y$, $z$ in $\mathfrak{g}$.
Due to the Cartan's criterion, there  always exists an invariant bilinear form on a semi-simple Lie algebra. Unfortunately, it is not true for semi-simple Lie superalgebras. Moreover, the following definition of metric Lie superalgebras can be viewed as a generalization of semi-simple Lie algebras to Lie superalgebras.

\begin{defn}\cite[Definition 1.10]{2}
A \emph{metric (or quadratic) Lie superalgebra} $(\mathfrak{g},B)$ is a Lie superalgebra $\mathfrak{g}$ with an even non-degenerate, supersymmetric, invariant bilinear form $B$. In this case, $B$ is called an invariant scalar product on $\mathfrak{g}$.
\end{defn}

\begin{defn}\cite[Definition 1.15]{2}
Two metric Lie superalgebras $(\mathfrak{g},B)$ and $(\mathfrak{g}',B')$ are called \emph{isometrically isomorphic} (or \emph{i-isomorphic}) if there exists  a Lie superalgebra isomorphism $f: \mathfrak{g}\rightarrow \mathfrak{g}'$ satisfying $B'(f(x),f(y))=B(x,y)$ for all $x,y\in \mathfrak{g}$. In this case, we write
 $\mathfrak{g}\mathop{\simeq}\limits^{i}\mathfrak{g}'$.
\end{defn}

The even part of a metric Lie superalgebra is  a metric Lie algebra, which is a Lie algebra with a non-degenerate, symmetric, invariant bilinear form. In fact, the characterization for metric Lie superalgebras can be reduced to their metric Lie algebras by the following lemma.

\begin{prop}\cite[Proposition 2.9]{B-B} \label{dim}
A Lie superalgebra $\mathfrak{g}$ is metric if and only if $\mathfrak{g}_{\bar{0}}$ is a metric Lie algebra with respect to a bilinear form $B_0$ and there exists a skew-symmetric non-degenerate   bilinear form $B_{1}$ on $\mathfrak{g}_{\bar{1}}$ such that, for all $x,y \in \mathfrak{g}_{\bar{1}}, z\in \mathfrak{g}_{\bar{0}}$,

(2) $B_0([x,y],z)=B_{1}(x,[y,z])$,

(1) ($\mathfrak{g}_{\bar{0}}$-invariant) $B_1([z,x],y)=-B_1(x,[z,y])$.
\end{prop}

\begin{cor}\label{odd}
From this Proposition, it follows that $\mathrm{dim}\ \mathfrak{g}_{\bar{1}}$  is even.
\end{cor}

Let $(\mathfrak{g},B)$ be a metric Lie superalgebra. An ideal $I$ of $\mathfrak{g}$ is called non-degenerate if $B|_{I\times I}$ is non-degenerate. Obviously, if $I$ is a non-degenerate ideal, then $(I,B|_{I\times I})$ is also a metric Lie superalgebra. It is known that any semi-simple Lie algebra can be decomposed into the direct sum of its simple ideals. The following proposition is analogous to what happens in the semi-simple Lie algebra case.

\begin{prop}\cite[Proposition 2.6]{B-B}\label{de}
 Let $(\mathfrak{g},B)$ be a metric Lie superalgebra. Then
 $$\mathfrak{g}=\bigoplus\limits_{i=1}^{r}\mathfrak{g}_{i},$$
such that, for all $1\leq i\leq r$,

(1) $\mathfrak{g}_{i}$ is a non-degenerate ideal of $\mathfrak{g}$.

(2) $\mathfrak{g}_{i}$ contains no nontrivial non-degenerate ideal of $\mathfrak{g}$.

(3) $B(\mathfrak{g}_{i},\mathfrak{g}_{j})=0$ for all $i\neq j$.
\end{prop}
\begin{defn}
In the above proposition, if $r=1$, $\mathfrak{g}$ is called \emph{indecomposable}. Otherwise, $\mathfrak{g}$ is called \emph{decomposable}.
\end{defn}

\begin{lem}\label{abel}
For a metric Lie superalgebra $\mathfrak{g}$, if $\mathrm{dim}\ \mathfrak{g}_{\bar{0}}\leq 1$, then $\mathfrak{g}$ is Abelian.
\begin{proof}
It is sufficient to prove the Lemma in the 1-dimensional even part case. Suppose that $(\mathfrak{g},B)$ is a metric Lie superalgebra, spanned by
$$\{e_1\mid e_2,\cdots,e_n\}.$$
Because of the invariance of $B$, we obtain
\begin{equation*}
B(e_{1}, [\mathfrak{g}_{\bar{1}},\mathfrak{g}_{\bar{1}}])=B([e_{1}, \mathfrak{g}_{\bar{1}}],\mathfrak{g}_{\bar{1}}).
\end{equation*}
Thus, $[\mathfrak{g}_{\bar{1}},\mathfrak{g}_{\bar{1}}]=0$ if and only if $[e_{1}, \mathfrak{g}_{\bar{1}}]=0$.
Suppose that  $[e_i,e_j]=k_{ij}e_{1}$ for any $2\leq i, j\leq n$.
In order to prove $\mathfrak{g}$ is Abelian, it is sufficient to prove $k_{ij}=0$.  At first, we claim that $k_{ii}=0$.
 By the super Jacobi identity, for any $2\leq i\leq n$,
$[[e_i,e_i],e_i]=0$. If there exists  $i_{0}$ such that $k_{i_{0}i_{0}}\neq 0$, then $[e_{1},e_{i_{0}}]=0$.
From
\begin{equation*}
  B([e_1,e_{i_{0}}],e_{i_{0}})=B(e_1,[e_{i_{0}},e_{i_{0}}])=0,
\end{equation*}
we obtain that
$B(e_1,e_1)=0$. It is a contradiction to the non-degenerate property of $B$.
In addition, we  obtain that for $2\leq i\neq j\leq n$, $[[e_i,e_j],e_j]=0$ from the super Jacobi identity of $e_i$, $e_j$ and $e_j$. If there exist $i_{1}$, $j_{1}$ such that $k_{i_{1},j_{1}}\neq 0$, then we have $[e_1,e_{j_{1}}]=0$. From
\begin{equation*}
  B([e_1,e_{j_{1}}],e_{i_{1}})=B(e_1,[e_{j_{1}},e_{i_{1}}])=0,
\end{equation*}
we obtain that
$B(e_1,e_1)=0$. It is a contradiction to the non-degenerate property of $B$. The proof is complete.
\end{proof}
\end{lem}

\subsection{Double extension for metric Lie superalgebras}

We introduce the theory of  double extension in this section, which is a important method to construct metric Lie superalgebras. For more details, readers are referred to \cite{B-B,2,suvey,F-P,class}.
\begin{defn}\cite[Chapter \Rmnum{1},1.4]{Kac}
Let $\mathfrak{g}$ be a Lie superalgebra and $D$ a homogeneous linear transformation of $\mathfrak{g}$. $D$ is called a \emph{derivation} of $\mathfrak{g}$ if
$$D([x,y])=[D(x),y]+(-1)^{|D||x|}[x,D(y)],\ x,y\in\mathfrak{g}.$$
\end{defn}

\begin{defn}\cite[Definition 3.4]{B-B}
Let $\mathfrak{g}$ be a Lie superalgebra and $B$ a bilinear form on $\mathfrak{g}$. Let $D$ be a derivation of $\mathfrak{g}$. $D$ is called \emph{skew-supersymmetric} with respect to $B$ if
$$B(D(x),y)=-(-1)^{|D||x|}B(x,D(y)),\ x,y\in\mathfrak{g}.$$
\end{defn}
Denote by $\mathrm{Der}(\mathfrak{g})$ and $\mathrm{Der}(\mathfrak{g},B)$ the derivation space of $\mathfrak{g}$ and the
skew-supersymmetric derivation space of $\mathfrak{g}$ with respect to $B$, respectively. By the above definitions, $\mathrm{Der}(\mathfrak{g})$ and $\mathrm{Der}(\mathfrak{g},B)$ are both Lie-super subalgebras of the general linear superalgebra $\mathfrak{gl}$.

\begin{thm}\cite[Theorem 1]{B-B}\label{B-B}
Let $(\mathfrak{g}_{1},B_{1})$ be a metric Lie superalgebra, $\mathfrak{g}_2$ a Lie superalgebra and $\psi:
\mathfrak{g}_2\rightarrow \mathrm{Der}(\mathfrak{g}_{1},B_1)\subseteq  \mathrm{Der}(\mathfrak{g}_{1})$ a morphism of Lie superalgebras.
Let $\varphi$ be the linear mapping from $\mathfrak{g}_{1}\times \mathfrak{g}_{1}$ to $\mathfrak{g}_{2}^{\ast}$, defined by
$$\varphi(x,y)(z)=(-1)^{(|x|+|y|)|z|}B_{1}(\psi(z)(x),y),\ x,y\in \mathfrak{g}_{1}, z\in \mathfrak{g}_{2}.$$
Let $\pi$ be the coadjoint representation of $\mathfrak{g}_{2}$. The the vector space $\mathfrak{g}=\mathfrak{g}_{2}\bigoplus \mathfrak{g}_{1}\bigoplus \mathfrak{g}_{2}^{\ast}$ with the products
\begin{eqnarray*}
&&[x_2,y_2]_\mathfrak{g}=[x_2,y_2]_{\mathfrak{g}_2},\quad [x_2,y_1]_\mathfrak{g}=\psi(x_2)y_1,\quad [x_2,f]_{\mathfrak{g}}=\pi(x_2)f,   \\
 &&[x_1,y_1]_\mathfrak{g}=[x_1,y_1]_{\mathfrak{g}_1}+\varphi(x_1,y_1),
\end{eqnarray*}
where $x_1, y_1\in \mathfrak{g}_{1}$, $x_2, y_2\in\mathfrak{g}_{2}$, $f\in \mathfrak{g}_{2}^{\ast}$, is a Lie superalgebra.
Moreover, if $B_2$ is an even, supersymmetric, invariant (not necessary non-degenerate) bilinear form on $\mathfrak{g}_{2}$, then the bilinear form $B_\mathfrak{g}$, defined on $\mathfrak{g}$ by
$$B_{\mathfrak{g}}(x_2, y_2)=B_{2}(x_2, y_2),\quad  B_{\mathfrak{g}}(f, y_2)=f(y_2),\quad B_{\mathfrak{g}}(x_{1},y_{1})=B_{1}(x_{1},y_{1}),$$
where $x_1, y_1\in \mathfrak{g}_{1}$, $x_2, y_2\in\mathfrak{g}_{2}$,  $f\in \mathfrak{g}_{2}^{\ast}$, is an invariant scalar product on $\mathfrak{g}$.
\end{thm}
\begin{defn}
In the above Theorem, the metric Lie superalgebra $(\mathfrak{g},B_{\mathfrak{g}})$ is called the \emph{double extension} of $(\mathfrak{g}_1,B_1)$ by $\mathfrak{g}_2$ by means of $\psi$. In particular, if $\mathfrak{g}_2$ is a 1-dimensional Lie algebra, $(\mathfrak{g},B_{\mathfrak{g}})$ is called the \emph{1-dimensional double extension} of $(\mathfrak{g}_1,B_1)$ by means of $\psi$.
\end{defn}

\begin{cor}\label{double}
Suppose  $(\mathfrak{g}_{1},B_{1})$ is  a metric Lie superalgebra and $D\in\mathrm{Der}_{\bar{0}}\ (\mathfrak{g}_{1},B_1)$. Let $B_{\C D}$ be an invariant symmetric (not necessary non-degenerate) bilinear form on $\C D=\langle D\rangle$. Then the 1-dimensional double extension $(\mathfrak{g}=\C D \bigoplus\mathfrak{g}_{1}\bigoplus \C D^{\ast},B_{\mathfrak{g}})$ of $(\mathfrak{g}_1,B_1)$ by means of $D$ is given as follows:

(1) The Lie-super brackets on $\mathfrak{g}$ are given by
$$[x,y]_{\mathfrak{g}}=[x,y]_{\mathfrak{g}_{1}}+B_{1}(D(x),y)D^{\ast}, \quad [D,x]_{\mathfrak{g}}=D(x),\ x,y\in \mathfrak{g}_{1}.$$

(2) The invariant scalar product $B_\mathfrak{g}$ is given by
$$B_{\mathfrak{g}}(D, D)=B_{\C D}(D, D),\quad  B_{\mathfrak{g}}(D^{\ast}, D)=1,\quad B_{\mathfrak{g}}(x,y)=B_{1}(x,y),\ x,y\in \mathfrak{g}_{1}.$$
\end{cor}
\begin{proof}
It follows from the Theorem \ref{B-B}.
\end{proof}
Double extension is an important method to construct examples of metric Lie superalgebras. In Lie case, the following
result allows us to inductively classify  metric Lie algebras by double extension.
\begin{thm}\cite[Theorem 1]{M-R}
A non-simple metric Lie algebra which is indecomposable is a double extension of some  metric Lie algebra by a 1-dimensional Lie algebra or a simple Lie algebra.
\end{thm}
The situation for metric Lie superalgebras is more complicated. It is still an open question whether every indecomposable non-simple metric Lie superalgebra $(\mathfrak{g},B)$ can be obtained by double extension. However, in the cases $\mathrm{dim}\ \mathfrak{g}_{\bar{0}}=2$ and $\mathrm{dim}\ \mathfrak{g}_{\bar{1}}=2$,
the answer is positive.

\begin{thm}\cite[Theorem 3]{B-B}\label{d1}
Any indecomposable non-simple metric Lie superalgebra $(\mathfrak{g}=\mathfrak{g}_{\bar{0}}\bigoplus \mathfrak{g}_{\bar{1}},B)$ with $\mathrm{dim}\ \mathfrak{g}_{\bar{1}}=2$, is a double extensions by a 1-dimensional or semi-simple Lie algebra.
\end{thm}

\begin{prop}\cite[Propsition 4.8]{2}\label{d2}
Let $\mathfrak{g}$ be a non-Abelian metric Lie superalgebra with 2-dimensional even part. Then $\mathfrak{g}$ is a double extension of the symplectic space
$\mathfrak{g}_{\bar{1}}$ (regarded as an Abelian Lie
superalgebra) by a 1-dimensional Lie algebra.
\end{prop}
\begin{exa}
We construct  in detail the 1-dimensional double extensions of $\mathfrak{\C}_{0|2}$.
 Let $\mathfrak{\C}_{0|2}=\mathrm{span}\{h_1,h_2\}$ be the symplectic space with the symplectic form given by
$$ B_{\mathfrak{\C}_{0|2}}=\left(
   \begin{array}{cc}
     0 & 1 \\
     -1 & 0 \\
   \end{array}
 \right).$$
 Denote by $\mathfrak{sp}(\mathfrak{\C}_{0|2},B_{\mathfrak{\C}_{0|2}})$ the 3-dimensional symplectic Lie algebra, which consists of all linear transformations of $\mathfrak{\C}_{0|2}$ which are skew-symmetric with respect to
 $B_{\mathfrak{\C}_{0|2}}$ \cite{H}. Then we have
 $\mathrm{Der}_{\bar{0}}\ (\mathfrak{\C}_{0|2},B_{\mathfrak{\C}_{0|2}})=\mathfrak{sp}(\mathfrak{\C}_{0|2},B_{\mathfrak{\C}_{0|2}})$,
 spanned by
$$D_1=\left(
  \begin{array}{cc}
    1 & 0 \\
    0 & -1 \\
  \end{array}
\right),\quad
 D_2=\left(
  \begin{array}{cc}
    0 & 1 \\
    0 & 0 \\
  \end{array}
\right),\quad
D_3=\left(
  \begin{array}{cc}
    0 & 0 \\
    1 & 0 \\
  \end{array}
\right).$$
 Let $D=k_1D_1+k_2D_2+k_3D_3$, $k_{1}, k_{2},k_3\in \C$ and $\mathfrak{g}=\C D \bigoplus\mathfrak{\C}_{0|2}\bigoplus \C D^{\ast}$ be a double extension of $(\mathfrak{\C}_{0|2},B_{\mathfrak{\C}_{0|2}})$ by $\C D$.
 From Corollary \ref{double}, the brackets of $\mathfrak{g}$ are given by
 \begin{eqnarray*}
 &&[h_1,h_1]=-k_3 D^{\ast},\quad [h_1,h_2]=k_1 D^{\ast},\quad [h_2,h_2]=k_2 D^{\ast} \\
&& [D,h_1]=k_1 h_1+k_3 h_2,\quad [D,h_2]=-k_1 h_2+k_2 h_1.
 \end{eqnarray*}
Then we determine the brackets of $\mathfrak{g}$ in the following cases:

\emph{Case 1}: $k_1=k_2=k_3=0$. Then $\mathfrak{g}$ is the Abelian Lie superalgebra $\C_{2|2}$.

\emph{Case 2}: $k_1=0$, $k_2,k_3\neq 0$. We choose a basis of $\mathfrak{g}$ as follows:
$$e_1=-2k_3 D^{\ast},\quad e_2=\frac{1}{\sqrt[]{k_2k_3}}D,\quad e_3=h_1+\sqrt[]{\frac{k_3}{k_2}}h_2,\quad e_4=h_1-\sqrt[]{\frac{k_3}{k_2}}h_2.$$
Then the brackets are given by
$$[e_3,e_4]=e_1,\quad [e_2,e_3]=e_3,\quad [e_2,e_4]=-e_4.$$

\emph{Case 3}: $k_1=k_2=0$, $k_3\neq 0$. We choose a basis of $\mathfrak{g}$ as follows:
$$e_1=- D^{\ast},\quad e_2=D,\quad e_3=\sqrt[]{k_3}h_2,\quad e_4=\frac{1}{\sqrt[]{k_3}}h_1.$$
Then the brackets are given by
$$[e_4,e_4]=e_1,\quad [e_2,e_4]=e_3.$$

\emph{Case 4}: $k_1=k_3=0$, $k_2\neq 0$. We choose a basis of $\mathfrak{g}$ as follows:
$$e_1=k_2 D^{\ast},\quad e_2=D,\quad e_3=k_2h_1,\quad e_4=h_2.$$
Then the brackets are given by
$$[e_4,e_4]=e_1,\quad [e_2,e_4]=e_3.$$

\emph{Case 5}: $k_1\neq0$, $k_2=k_3=0$. We choose a basis of $\mathfrak{g}$ as follows:
$$e_1=k_1 D^{\ast},\quad e_2=\frac{1}{k_1}D,\quad e_3=h_1,\quad e_4=h_2.$$
Then the brackets are given by
$$[e_3,e_4]=e_1,\quad [e_2,e_3]=e_3,\quad [e_2,e_4]=-e_4.$$

\emph{Case 6}: $k_1\neq0$, $k_2=0$, $k_3\neq0$. We choose a basis of $\mathfrak{g}$ as follows:
$$e_1=\frac{2k^{2}_1}{k_3} D^{\ast},\quad e_2=\frac{1}{k_1}D,\quad e_3=\frac{2k_1}{k_3}h_1+h_2,\quad e_4=h_2.$$
Then the brackets are given by
$$[e_3,e_4]=e_1,\quad [e_2,e_3]=e_3,\quad [e_2,e_4]=-e_4.$$

\emph{Case 7}: $k_1, k_2\neq0$, $k_3=0$. We choose a basis of $\mathfrak{g}$ as follows:
$$e_1=-\frac{2k^{2}_1}{k_2} D^{\ast},\quad e_2=\frac{1}{k_1}D,\quad e_3=h_1,\quad e_4=h_1-\frac{2k_1}{k_2}h_2.$$
Then the brackets are given by
$$[e_3,e_4]=e_1,\quad [e_2,e_3]=e_3,\quad [e_2,e_4]=-e_4.$$

\emph{Case 8}: $k_1, k_2, k_3\neq0$, $k^{2}_1+k_2k_3=0$. We choose a basis of $\mathfrak{g}$ as follows:
$$e_1=k_2 D^{\ast},\quad e_2=\frac{1}{k_2}D,\quad e_3=h_1-\frac{k_1}{k_2}h_2,\quad e_4=h_2.$$
Then the brackets are given by
$$[e_4,e_4]=e_1,\quad [e_2,e_4]=e_3.$$
\end{exa}

\emph{Case 9}: $k_1, k_2, k_3\neq0$, $k^{2}_1+k_2k_3\neq0$. We choose a basis of $\mathfrak{g}$ as follows:
\begin{eqnarray*}
 &&e_1=-2\frac{k^{2}_1+k_2k_3}{k_2} D^{\ast},\quad e_2=\frac{1}{\sqrt[]{k^{2}_1+k_2k_3}}D, \\
  && e_3=h_1-\frac{k_1-\sqrt[]{k^{2}_1+k_2k_3}}{k_2}h_2,\quad e_4=h_1-\frac{k_1+\sqrt[]{k^{2}_1+k_2k_3}}{k_2}h_2.
\end{eqnarray*}
Then the brackets are given by
$$[e_3,e_4]=e_1,\quad [e_2,e_3]=e_3,\quad [e_2,e_4]=-e_4.$$
We summarize the results in the following Proposition.
\begin{prop}\cite[Proposition 3.3]{2}\label{2,2}
Suppose that $\mathfrak{g}$ is a $2|2$-dimensional non-Abelian metric Lie superalgebra. Then $\mathfrak{g}$ is isomorphic to the following ones:

$(1)\ \mathfrak{g}_{2|2}^{1}:\quad [e_4,e_4]=e_1,\ [e_2,e_4]=e_3$.

$(2)\ \mathfrak{g}_{2|2}^{2}:\quad [e_3,e_4]=e_1,\
[e_2,e_3]=e_3,\
[e_2,e_4]=-e_4.$
\end{prop}
\begin{proof}
It follows from Theorem \ref{d1} or Proposition \ref{d2}.
\end{proof}

\section{Cohomology and deformations}
In this section, we recall the definition of cohomology for Lie superalgebras. For more details, the reader is referred to \cite{BB}.
Suppose  $\mathfrak{g}$ is a Lie superalgebra and $M$ is  a $\mathfrak{g}$-module, the $n$-cochain space $\mathrm{C}^{n}(\mathfrak{g},M)$ is defined by
$$\mathrm{C}^{n}(\mathfrak{g},M)=\mathrm{Hom}(\wedge^{n}\mathfrak{g},M).$$
The coboundary operator $\mathrm{d}:\ \mathrm{C}^{n}(\mathfrak{g},M)\longrightarrow \mathrm{C}^{n+1}(\mathfrak{g},M)$ is defined by
\begin{eqnarray*}
\mathrm{d}(f)(x_{0},\cdots, x_{n})=&&\sum\limits_{i=0}^{n}(-1)^{b_{i}+|x_{i}||f|}x_i\cdot f(x_{0},\cdots,\widehat{x}_{i},\cdots,x_n) \\
&&+\sum\limits_{0\leq p<q\leq n}(-1)^{c_{p,q}}f([x_p,x_q],x_{0},\cdots,\widehat{x}_p,\cdots,\widehat{x}_{q},\cdots,x_n),
\end{eqnarray*}
where
\begin{eqnarray*}
&&b_{i}=i+|x_{i}|(|x_{0}|+\cdots+|x_{i-1}|), \\
&&c_{p,q}=p+q+(|x_{p}|+|x_{q}|)(|x_{0}|+\cdots+|x_{p-1}|)+|x_{q}|(|x_{p+1}|+\cdots+|x_{q-1}|),
\end{eqnarray*}
and\quad $\widehat{}$\quad denotes an omitted term. A standard fact is that $\mathrm{d}^{2}=0$. We can define the
\emph{$n$-cohomology} by
$$\mathrm{H}^{n}(\mathfrak{g},M)=\mathrm{Z}^{n}(\mathfrak{g},M)/\mathrm{B}^{n}(\mathfrak{g},M),$$
where
\begin{eqnarray*}
&& \mathrm{Z}^{n}(\mathfrak{g},M)=\mathrm{Ker}\ (\mathrm{d}: \mathrm{C}^{n}(\mathfrak{g},M)\longrightarrow \mathrm{C}^{n+1}(\mathfrak{g},M)), \\
 &&\mathrm{B}^{n}(\mathfrak{g},M)=\mathrm{Im}\ (\mathrm{d}: \mathrm{C}^{n-1}(\mathfrak{g},M)\longrightarrow \mathrm{C}^{n}(\mathfrak{g},M)).
\end{eqnarray*}
A cochain is called a \emph{cocycle} (resp. \emph{coboundary}) if it is in $\mathrm{Z}^{n}(\mathfrak{g},M)$
 $(\mathrm{resp}.\ \mathrm{B}^{n}(\mathfrak{g},M))$.
The cohomology theory has many applications in mathematics and physics \cite{,Fuks,Musson,Divided power,BB}. One of the most important applications of cohomology is computing formal deformations, for which we need the even part of the second cohomology with adjoint coefficients $\mathrm{H}^{2}(\mathfrak{g},\mathfrak{g})$.
\begin{defn}
A \emph{formal 1-parameter deformation} of a Lie superalgebra $(\mathfrak{g}, [\ ,\ ])$ (for associative algebras, see \cite{G}) is a family of Lie superalgebra structures on the $\mathbb{C}[\![t]\!]$-module $\mathfrak{g}[\![t]\!]=\mathfrak{g}\otimes_{\C}\mathbb{C}[\![t]\!]$ such that
$$[\ ,\ ]_{t}=[\ ,\ ]+\sum_{i=1}^{\infty} t^{i} \phi_{i}(\ ,\ )$$
where each $\phi_{i}$ is in $\mathrm{C}^{2}(\mathfrak{g}, \mathfrak{g})_{\bar{0}}$.
\end{defn}
\begin{remark}
By definition, a formal 1-parameter deformation of $\mathfrak{g}$ is a family of Lie superalgebra structures on $\mathfrak{g}$, parameterized by $\mathbb{C}[\![t]\!]$. For general theory, readers are referred to \cite{Fuks,BB,F,F1}.
\end{remark}
\begin{defn}
If the super Jacobi identity for $[\ ,\ ]_{t}$ holds up to order $n$, then $[\ ,\ ]_{t}$ is called a deformation of order $n$. In particular, if the super Jacobi identity for $[\ ,\ ]_{t}$ is satisfied only up to the $t$-term, then $[\ ,\ ]_{t}$ is called  a \emph{first order deformation} or \emph{infinitesimal deformation}.
\end{defn}
\begin{remark}
By definition, a deformation of order $n$ is a family of  Lie superalgebra structures on $\mathfrak{g}$, parameterized by $\mathbb{C}[\![t]\!]/(t^{n+1})$.
\end{remark}
\begin{defn}
Suppose that $[\ ,\ ]_{t}=\sum t^{i}\phi _{i}$ and $[\ ,\ ]'_{t}=\sum t^{i}\phi'_{i}$ are two formal 1-parameter deformations of $\mathfrak{g}$. $[\ ,\ ]_{t}$ and $[\ ,\ ]'_{t}$ are called \emph{equivalent} if there exists a linear isomorphism $\widehat{\psi}_{t}=\mathrm{id}_{\mathfrak{g}}+\psi_1 t+\psi_2 t^{2}+\cdots$, where $\psi_i$ is in $\mathrm{C}^{1}(\mathfrak{g},\mathfrak{g})_{\bar{0}}$, such that
$$\widehat{\psi}_{t}([x ,y]_{t})=[\widehat{\psi}_{t}(x),\widehat{\psi}_{t}(y)]'_t,\quad \mathrm{for}\ x,y \in \mathfrak{g}.$$
\end{defn}
\begin{defn}
A formal 1-parameter deformation is called \emph{trivial} if it is equivalent to the original bracket.  If every formal 1-parameter deformation  of $\mathfrak{g}$ is trivial, then $\mathfrak{g}$ is called \emph{rigid}.
\end{defn}
\begin{remark}

Suppose that $[\ ,\ ]_{t}=\sum t^{i}\phi _{i}$ and $[\ ,\ ]'_{t}=\sum t^{i}\phi'_{i}$ are two  formal 1-parameter deformations of $\mathfrak{g}$.
The skew-supersymmery for $[\ ,\ ]_{t}$  follows  from the fact that each $\phi_{i}$ is a cochain.
The bracket  $[\ ,\ ]_{t}$ is $\mathbb{Z}_{2}$-graded because  each $\phi_{i}$ is even.
 The super Jacobi identity implies that $\phi_{1}$ is indeed an even cocycle.
 More generally, if $\phi_{1}$ vanishes identically, the first non-vanishing $\phi_{i}$ will be an even cocycle.
If $[\ ,\ ]_{t}=\sum t^{i}\phi _{i}$ and $[\ ,\ ]'_{t}=\sum t^{i}\phi'_{i}$ are equivalent,
then there exists a linear isomorphism $\widehat{\psi}_{t}=\mathrm{id}_{\mathfrak{g}}+\psi_1 t+\psi_2 t^{2}+\cdots$.
Comparing the two sides, we obtain that
$$\phi_1-\phi'_1=d(\psi_1).$$
It means that every equivalence class of deformations defines uniquely a cohomology class of the even part of the 2-cohomology. We call the cocycle $\phi_1$ the \emph{infinitesimal part} of the  deformation. On the other hand, if
every even 2-cocycle of $\mathfrak{g}$ is a coboundary, then $\mathfrak{g}$ only has trivial deformations.
\end{remark}
\begin{thm}
The 2-cohomology classes are in 1-1 correspondence with non-equivalent infinitesimal deformations.
\end{thm}
\begin{thm}\label{r1}
If $\mathrm{H}^{2}(\mathfrak{g},\mathfrak{g})_{\bar{0}}=0$, then $\mathfrak{g}$ is rigid.
\end{thm}
Given a cohomology class $[\alpha]$, a natural question is whether there exists a formal 1-parameter deformation with the infinitesimal part being a representative of $[\alpha]$. In order to answer this question, we introduce the following definition.

\begin{defn}\cite{blue}
A complex $\mathcal{C=}(\mathcal{C}^{n},\mathrm{d})_{n=0}^{\infty}$ with an operation $[\ ,\ ]$ is called a
\emph{Lie $\mathbb{Z}$-graded superalgebra} if
\begin{eqnarray*}
&&[x,y]=-(-1)^{|x||y|+pq}[y,x], \\
&&\mathrm{d}([x,y])= [\mathrm{d}(x),y]+(-1)^{p}[x,\mathrm{d}(y)],\\
&&(-1)^{|x||z|+pr}[x,[y,z]]+(-1)^{|y||x|+qp}[y,[z,x]]+(-1)^{|z||y|+rq}[z,[x,y]]=0,
\end{eqnarray*}
for any $x\in \mathcal{C}^{p}, y \in\mathcal{C}^{q}, z\in\mathcal{C}^{r}$.
\end{defn}
For $\alpha\in \mathrm{C}^{p}(\mathfrak{g},\mathfrak{g}), \beta\in \mathrm{C}^{q}(\mathfrak{g},\mathfrak{g})$, the product $\alpha\beta\in\mathrm{C}^{p+q-1}(\mathfrak{g},\mathfrak{g})$ is defined by
\begin{eqnarray*}
(\alpha\beta)(x_{1},\ldots,x_{p+q-1}) &=& \sum\limits_{1\leq i_1<\cdots<i_{p-1}\leq p+q-1}(-1)^{a_{i_{1},\cdots,i_{p-1}}}\alpha(x_{i_1},\ldots,x_{i_{p-1}}, \\
   &&\beta(x_1,\ldots,\widehat{x}_{i_{1}},
\cdots,\widehat{x}_{i_{p-1}},\ldots,x_{p+q-1}))
\end{eqnarray*}
where
\begin{eqnarray*}
&& a_{i_{1},\cdots,i_{p-1}}=\sum\limits_{s=1}^{p-1}\left(i_s-s+a_{i_{1},\cdots,i_{p-1}}^{s}\right), \\
&& a_{i_{1},\cdots,i_{p-1}}^{s}=|x_{i_{s}}|\left(|\beta|+\sum\limits_{t\in\{1,\cdots,i_{s}-1\}\backslash\{i_1,\cdots,i_{s-1}\}}|x_{t}|\right).
\end{eqnarray*}
Define the bracket operation
$[\alpha,\beta]=\alpha\beta-(-1)^{|\alpha||\beta|+(p-1)(q-1)}\beta\alpha$.
\begin{thm}\cite{blue}
If we denote
$\mathcal{C}^{q}= \mathrm{C}^{q+1}(\mathfrak{g},\mathfrak{g})$ and
$\mathcal{H}^{q}= \mathrm{H}^{q+1}(\mathfrak{g},\mathfrak{g})$, then the above bracket  makes
$\mathcal{C}=\bigoplus_{q}\mathcal{C}^{q}$ and $\mathcal{H}=\bigoplus_{q}\mathcal{H}^{q}$
 Lie $\mathbb{Z}$-graded superalgebras.
 \end{thm}
\begin{lem}\cite[Lemma 3.3.1]{blue}\label{Ma1}
The 2-cochain sequence $\{\phi_{n}\}_{n=1}^{\infty}$ defines a formal deformation of $\mathfrak{g}$ if and only if the elements $\phi_{n}$ in $\mathrm{C}^{2}(\mathfrak{g},\mathfrak{g})_{\bar{0}}$ satisfy the equations
$$
  \mathrm{d}(\phi_{n})+\frac{1}{2}\sum_{\underset{i,j> 0}{i+j=n}}[\phi_{i},\phi_{j}]=0,\quad \mathrm{for}\ \mathrm{any}\ n\geq 1.
$$
\end{lem}
\begin{remark}
Suppose that $\phi_{1}$ is a cocycle. If $[\phi_{1},\phi_{1}]=0$, then a formal deformation with the infinitesimal part $\phi_{1}$ can be given by taking $\phi_{i}=0$  for all $i\geq 2$.
\end{remark}
\begin{defn}
A formal 1-parameter deformation is called

(1) a \emph{metric deformation} if it defines a  metric Lie superalgebra.

(2) a \emph{jump deformation} if for any non-zero value of the parameter $t$ near the origin, it gives isomorphic algebra (which is of course different from the original one).

(3) a \emph{smooth deformation} if for any different  non-zero values of the  parameter near the origin, it defines non-isomorphic algebras (by symmetry sometimes there can be coincidences).

\end{defn}
\begin{defn}
An infinitesimal deformation of $\mathfrak{g}$ is called

(1)  \emph{real} if no higher order terms appear in the super Jacobi identity.

(2)  \emph{metric} if it defines a  metric Lie superalgebra on $\mathfrak{g}$, parameterized by $\C[\![t]\!]/(t^{2})$.
 \end{defn}
 \begin{remark}
 By definition, an infinitesimal deformation is real if and only if the bracket of its infinitesimal part is zero.
 \end{remark}
 \begin{remark}
 Note that the corresponding infinitesimal deformation of a metric deformation is a metric infinitesimal deformation, which allows us to construct a metric deformation by starting  with a metric infinitesimal deformation.
  In particular, if this metric infinitesimal deformation is real (i.e., it defines a  metric Lie superalgebra), then  it gives a metric deformation without any further obstruction (for more examples, see \cite{F}).
 \end{remark}
We close this section with the discussions for two special cases, which are Abelian Lie superalgebras and simple Lie superalgebras with non-degenerate Killing forms.
Note that an Abelian Lie superalgebra with an even-dimensional odd part is always metric by  Proposition \ref{dim}.  It is easy to see that Abelian Lie superalgebras deform everywhere because of the triviality of coboundary.
Another special case is  simple Lie superalgebras with non-degenerate Killing forms. They only have trivial deformations because of the triviality of cohomology according to the following Theorem.
\begin{thm}\cite[Theorem 3]{Leites}\label{Leites}
Let $\mathfrak{g}$ be a semi-simple Lie superalgebra over a field of characteristic zero. Suppose that $M$ is a $\mathfrak{g}$-module defined by a nontrivial
irreducible representation $\rho$ and  the super trace form $T(\rho(X)\rho(Y))$, corresponding to this representation  is
non-degenerate. Then $\mathrm{H}^{n}(\mathfrak{g}, M)=0$ for any $n$.
\end{thm}
\begin{remark}
A basic fact is that every semi-simple Lie algebra is rigid \cite[Theorem 24.1]{white1}. Although it is not the case for semi-simple Lie superalgebras,
any finite-dimensional Lie superalgebra with non-degenerate Killing form can be decomposed into a direct sum of a semi-simple Lie algebra and classical simple Lie superalgebras  \cite{Kac,Musson}.
\end{remark}

\section{Low-dimensional metric Lie superalgebras and their deformations}
From now on, we suppose that $\mathfrak{g}_{m|n}$ is a $m|n$-dimensional non-Abelian metric Lie superalgebra, spanned by
$$\{e_{1},\cdots,e_m \mid e_{m+1},\cdots,e_{m+n}\},$$
where $e_1,\cdots,e_m$ are even and $e_{m+1},\cdots e_{m+n}$ are odd. Define $e^{i,j}_{k}\in\mathrm{C}^{2}(\mathfrak{g}_{m|n},\mathfrak{g}_{m|n})$ by $e^{i,j}_{k}:\ (e_i,e_j)\mapsto e_k$, $1\leq i,j,k\leq m+n$.

The classification and metric deformations of the metric Lie algebras with dimension $\leq 6$ have been studied in \cite{F-P}. In particular, there are only $\mathfrak{sl}(2,\mathbb{C})$ and the diamond Lie algebra $\mathfrak{b}$ which are indecomposable metric Lie algebras  with dimension $\leq 4$.
\begin{lem}\label{c}
Suppose that $\mathfrak{g}$ is a metric Lie superalgebra.

(1) If $\mathfrak{g}_{\bar{0}}$ is $2$-dimensional, then $\mathfrak{g}_{\bar{0}}\cong \C_{2|0}$.

(2) If $\mathfrak{g}_{\bar{0}}$ is $3$-dimensional, then $\mathfrak{g}_{\bar{0}}\cong\mathfrak{sl}(2,\mathbb{C})$.

(3) If $\mathfrak{g}_{\bar{0}}$ is $4$-dimensional, then $\mathfrak{g}_{\bar{0}}\cong\mathfrak{sl}(2,\mathbb{C})\bigoplus \mathbb{C}$ or $\mathfrak{b}$.
\end{lem}
Suppose that  $\mathfrak{g}$ is a non-Abelian metric Lie superalgebra with dimension $\leq 6$ which has a non-zero odd part. Then $\mathrm{dim}\ \mathfrak{g}_{\bar{0}}=2$ or $\mathrm{dim}\ \mathfrak{g}_{\bar{1}}=2$
by Corollary \ref{odd} and Lemma \ref{abel}.
Consequently, all indecomposable metric Lie superalgebras which have non-zero odd parts have been classified by 1-dimensional double extension in \cite{class}.
By Corollary \ref{odd} and Lemma \ref{abel}, these metric Lie superalgebras are $2|2$, $3|2$, $4|2$ and $2|4$ dimensional.
In this section, we study nontrivial metric deformations of these metric Lie superalgebras case by case and we also point out jump and smooth deformations among them.

\subsection{$2|2$-dimensional metric Lie superalgebras.} There are  two indecomposable metric Lie superalgebras:
 $\mathfrak{g}_{2|2}^{1}$ and $\mathfrak{g}_{2|2}^{2}$ with nontrivial brackets \cite{class}:
 $$
 \begin{tabular}{lrl}
   (1) & $\mathfrak{g}_{2|2}^{1}$: & $[e_4,e_4]=e_1,\ [e_2,e_4]=e_3,$ \\
   (2) & $\mathfrak{g}_{2|2}^{2}$: & $[e_3,e_4]=e_1,\
[e_2,e_3]=e_3,\
[e_2,e_4]=-e_4$. \\
 \end{tabular}
 $$
We already have obtained  the classification by double extension in Proposition \ref{2,2}.
\begin{thm}\label{1}
The algebra $\mathfrak{g}_{2|2}^{1}$ has only one  metric deformation, which is a jump deformation to $\mathfrak{g}_{2|2}^{2}$.
\end{thm}
\begin{proof}
The even part of the 2-cohomology space is 4-dimensional, spanned by the  representative  even cocycles:
$$f_{1}=e^{1,2}_{1}-\frac{1}{2}e^{3,4}_{1},\
f_{2}=e^{1,2}_{2}+e^{1,3}_{3}-\frac{1}{2}e^{3,4}_{2},\
f_{3}=e^{2,3}_{3},\
f_{4}=e^{2,3}_{4}-e^{3,3}_{1}.$$
Let $[\ ,\ ]_t$ be the metric infinitesimal deformation defined by $\sum_{i=1}^{4}a_if_i$ ($a_i\in\C$). Then  $(\mathfrak{g}_{2|2}^{1},[\ ,\ ]_t)$, parameterized by $\C[\![t]\!]/(t^{2})$, is a metric Lie superalgebra w.r.t an invariant scalar product $B$. We have the  even part $((\mathfrak{g}_{2|2}^{1})_{\bar{0}},[\ ,\ ]_t)\cong\C_{2|0}$ by Lemma \ref{c} (1). So $a_{1}=a_2=0$. Since the invariance of $B$, one has
$$B([e_2,e_3]_t,e_4)=B(e_2,[e_3,e_4]_t)=0.$$
Thus $a_3=0$.
From $[f_4,f_4]=0$,
 the cocycle $a_4f_4$ ($a_4\neq 0$) defines a metric and real infinitesimal deformation isomorphic to $\mathfrak{g}_{2|2}^{2}$ via the change of basis:
$$
e'_1=-2a_4te_1,\ e'_{2}=\frac{1}{\sqrt{a_4t}}e_2,\ e'_3=e_3+\sqrt[]{a_{4}t}e_4,\ e'_4=e_3-\sqrt[]{a_{4}t}e_4.
$$
\end{proof}

\begin{thm}
The algebra $\mathfrak{g}_{2|2}^{2}$ only has trivial metric deformation.
\end{thm}
\begin{proof}
The even part of  the 2-cohomology space is 1-dimensional, spanned by the  representative even cocycle
$f=-e^{1,2}_1+e^{2,4}_4$.
Let $[\ ,\ ]_t$ be the metric infinitesimal deformation defined by $af$ ($a\in\C$). By Lemma \ref{c} (1), we have  $((\mathfrak{g}_{2|2}^{2})_{\bar{0}},[\ ,\ ]_t)\cong \C_{2|0}$, parameterized by $\C[\![t]\!]/(t^{2})$. So $a=0$. Therefore it only has trivial metric deformations.
\end{proof}

\subsection{$3|2$-dimensional metric Lie superalgebras.} There is only one indecomposable metric Lie superalgebra $\mathfrak{osp}(1,2)$ with nontrivial brackets \cite{class}:
\begin{eqnarray*}
&&[e_1,e_2]=e_3,\ [e_1,e_3]=-2e_1,\ [e_2,e_3]=2e_2,\ [e_1,e_5]=-e_4,\ [e_2,e_4]=-e_5,\\
&&[e_3,e_4]=e_4,\ [e_3,e_5]=-e_5,\ [e_4,e_4]=\frac{1}{2}e_1,\ [e_4,e_5]=\frac{1}{4}e_3, \ [e_5,e_5]=-\frac{1}{2}e_2.
\end{eqnarray*}
\begin{thm}
The algebras  $\mathfrak{osp}(1,2)$ only has trivial metric deformation.
\end{thm}
\begin{proof}
Since the Killing form of $\mathfrak{osp}(1,2)$ is non-degenerate \cite{Kac}, $\mathfrak{osp}(1,2)$ is rigid by Theorem \ref{r1} and Theorem \ref{Leites}.
\end{proof}

\subsection{$4|2$-dimensional metric Lie superalgebras.} There are the following indecomposable metric Lie superalgebras:
$\mathfrak{g}_{4|2}^{1}$, $\mathfrak{g}_{4|2}^{2}(\lambda)$ $(\lambda\neq 0)$, and $\mathfrak{g}_{4|2}^{3}$ with nontrivial brackets \cite{class}:
$$
\begin{tabular}{lrl}
  (1) & $\mathfrak{g}_{4|2}^{1}$ : &$[e_1,e_2]=e_2$,\ $[e_1,e_3]=-e_3$, $[e_2,e_3]=e_4$,\\
      &                            &$[e_1,e_6]=e_5$,\ $[e_6,e_6]=e_4$, \\
  (2) & $\mathfrak{g}_{4|2}^{2}(\lambda)$ $(\lambda\neq 0)$ : & $[e_1,e_2]=e_2$,\ $[e_1,e_3]=-e_3$, $[e_2,e_3]=e_4$,\\ & &$[e_1,e_5]=\lambda e_5$,\ $[e_1,e_6]=-\lambda e_6$,\ $[e_5,e_6]=\lambda e_4$,  \\
  (3) & $\mathfrak{g}_{4|2}^{3}$ : & $[e_1,e_2]=e_2$,\ $[e_1,e_3]=-e_3$, $[e_2,e_3]=e_4$,\\ & &$[e_1,e_5]=\frac{1}{2}e_5$,\ $[e_1,e_6]=-\frac{1}{2}e_6$,\ $[e_2,e_6]=e_5$, \\
  &   &$[e_5,e_6]=\frac{1}{2}e_4$,\ $[e_6,e_6]=e_3$.
\end{tabular}
$$
Here $\mathfrak{g}_{4|2}^{2}(\lambda_{1})\mathop{\simeq}\limits^{i} \mathfrak{g}_{4|2}^{2}(\lambda_{2})$ if and only if $\lambda_{1}=\lambda_{2}$. However, $\mathfrak{g}_{4|2}^{2}(\lambda)$ is isomorphic to $\mathfrak{g}_{4|2}^{2}(-\lambda)$ via the change of basis:
$$e'_1=-e_1,\ e'_2=e_3,\ e'_3=e_2,\ e'_4=-e_4,\ e'_5=e_5,\ e'_6=e_6.$$
\begin{thm}
The algebra $\mathfrak{g}_{4|2}^{1}$ has only one metric deformation, which is a smooth deformation to the family $\mathfrak{g}_{4|2}^{2}(\lambda)$ around $\lambda=0$.
\end{thm}
\begin{proof}
The even part of the  2-cohomology space is  3-dimensional, spanned by the  representative  even cocycles:
$$f_1=e^{1,5}_5,\ f_2=e^{1,5}_6-e^{5,5}_4,\ f_3=e^{1,2}_2+e^{1,4}_4+\frac{1}{2}e^{5,6}_4.$$
If the cocycle $\sum_{i=1}^{3}a_if_i$ defines a metric infinitesimal deformation, then we have $a_1=a_3=0$. From $[f_2,f_2]=0$, the cocycle $a_2f_2$ ($a_2\neq0$) defines a metric and real infinitesimal deformation isomorphic to $\mathfrak{g}_{4|2}^{2}(\sqrt{a_2t})$ via the change of basis:
$$e'_1=e_1,\ e'_2=e_2,\ e'_3=e_3,\ e'_4=e_4,\ e'_5=-\frac{1}{2 \sqrt{a_2t}}e_5-\frac{1}{2}e_6,\ e'_6=e_5-\sqrt{a_2t}e_6.$$

\end{proof}
\begin{thm}\label{dis}
Metric deformations of the family $\mathfrak{g}_{4|2}^{2}(\lambda)$ $(\lambda\neq0)$ follow two different patterns:

(1) The generic element for $\lambda\neq \pm \frac{1}{2}$  has one metric deformation, which is a smooth  metric  deformation to the family $\mathfrak{g}_{4|2}^{2}(\lambda)$ around itself.

(2) The special element $\mathfrak{g}_{4|2}^{2}(\frac{1}{2})\cong\mathfrak{g}_{4|2}^{2}(-\frac{1}{2})$
has a smooth  metric deformation to the family $\mathfrak{g}_{4|2}^{2}(\lambda)$ around itself and it has jump  metric deformations to $\mathfrak{g}_{4|2}^{3}$ and $\mathfrak{osp}(1,2)\oplus \C_{1|0}$.
\end{thm}
\begin{proof}
Note that the case $\lambda=0$ is generally excluded from this family, because $\mathfrak{g}_{4|2}^{2}(0)=\mathfrak{b}\bigoplus \C_{0|2}$. Beside that, in the cases $\lambda=\pm\frac{1}{2}$, the cohomology and deformation patterns are not generic.

(1)
The even part of the  2-cohomology space is  2-dimensional, spanned by the  representative  even cocycles:
$$f_1=e^{1,5}_5-e^{1,6}_6,\ f_2=e^{1,2}_2+e^{1,4}_4+e^{1,6}_6.$$
If the cocycle $\sum_{i=1}^{2}a_if_i$ defines a metric infinitesimal deformation, then we have
$a_2=0$. From $[f_1,f_1]=0$, the cocycle $a_1f_1$ ($a_1\neq 0$) defines a metric and real infinitesimal deformation isomorphic to $\mathfrak{g}_{4|2}^{2}(\lambda+a_1t)$ via the change of basis:
$$e'_1=e_1,\ e'_2=\lambda e_2 ,\ e'_3=\frac{1}{\lambda+a_1t}e_3 ,\ e'_4=\frac{\lambda}{\lambda+a_1t}e_4,\ e'_5=e_5 ,\ e'_6=e_6.$$

(2)
Because of the isomorphism, it is sufficient to consider $\mathfrak{g}_{4|2}^{2}(\frac{1}{2})$.
The even part of the  2-cohomology space is  4-dimensional, spanned by the  representative  even cocycles:.
 $$f_1=e^{1,5}_5-e^{1,6}_6,\ f_2=e^{1,2}_2+e^{1,4}_4+e^{1,6}_6,\ f_{3}=e^{2,6}_5+e^{6,6}_3,\  f_{4}=e^{3,5}_6-e^{5,5}_2.$$
 If the cocycle $\sum_{i=1}^{4}a_if_i$ defines a metric infinitesimal deformation, then we have that
$a_i=0$, $i\neq 1$, or $a_1=a_2=0$. We determine nontrivial metric deformation in the following cases:

\emph{Case 1}: $a_i=0$, $i\neq 1$. The cocycle $a_1f_1$ ($a_1\neq 0$) leads a smooth  metric  deformation around itself (see (1)).

\emph{Case 2}: $a_1=a_2=a_3=0$. The cocycle $a_4f_4$ ($a_4\neq 0$) defines a metric and real infinitesimal deformation isomorphic to $\mathfrak{g}_{4|2}^{3}$ via the change of basis:
$$e'_1=-e_1,\ e'_2=e_3,\ e'_3=-a_4te_2,\ e'_4=a_4te_4,\ e'_5=a_4te_6,\ e'_6=e_5.$$

\emph{Case 3}: $a_1=a_2=a_4=0$. The cocycle $a_3f_3$ ($a_3\neq 0$) defines a metric and real infinitesimal deformation isomorphic to $\mathfrak{g}_{4|2}^{3}$ via the change of basis:
$$e'_1=e_1,\ e'_2=e_2,\ e'_3=a_3te_3,\ e'_4=a_3te_4,\ e'_5=a_3te_5,\ e'_6=e_6.$$

\emph{Case 4}: $a_1=a_2=0,a_3a_4\neq0$. The infinitesimal deformation, defined by the cocycle $a_3f_3+a_4f_4$, is not real, since $[a_3f_3+a_4f_4,a_3f_3+a_4f_4]\neq 0$. However, it can be extended to a metric deformation. An example can be given by taking
$$\phi_1=a_3f_3+a_4f_4,\ \phi_2=2a_3a_4e^{2,3}_1+a_3a_4e^{5,6}_1.$$
Then we have
$$\mathrm{d}(\phi_2)=-\frac{1}{2}[\phi_1,\phi_1],\ [\phi_1,\phi_2]=[\phi_2,\phi_2]=0.$$
Moreover,
the bracket
$$[\ ,\ ]^{1}_t=[\ ,\ ]+t\phi_1(\ ,\ )+t^{2}\phi_2(\ ,\ )$$
defines a metric Lie superalgebra isomorphic to $\mathfrak{osp}(1,2)\oplus \C_{1|0}$ via the change of basis:
\begin{eqnarray*}
&&e'_1=\frac{1}{a_3t}e_2,\ e'_2= \frac{1}{a_4t}e_3,\ e'_3=\frac{1}{a_3a_4t^{2}}e_4+2e_1,\ e'_4=\frac{1}{a_3a_4}e_4, \\
&&e'_5=\frac{1}{t}\sqrt{-\frac{1}{2a_3a_4}}e_5,\ e'_6=-\frac{1}{t}\sqrt{-\frac{1}{2a_3a_4}}e_6.
\end{eqnarray*}
\end{proof}

\begin{thm}
The algebra $\mathfrak{g}_{4|2}^{3}$ has only one metric deformation, which is a jump  deformation to $\mathfrak{osp}(1,2)\bigoplus \C_{1|0}$.
\end{thm}
\begin{proof}
The even part of the  2-cohomology space is  2-dimensional, spanned by the  representative  even cocycles:
$$f_1=e^{1,2}_2+e^{1,4}_4+e^{1,5}_5,\   f_2=e^{2,3}_{1}+\frac{1}{2}e^{3,5}_6-\frac{1}{2}e^{5,5}_2+\frac{1}{2}e^{5,6}_1.$$
If the cocycle $\sum_{i=1}^{2}a_if_i$ defines a metric infinitesimal deformation, then we have
$a_1=0$. From $[f_2,f_2]=0$, the cocycle $a_2f_2$ ($a_2\neq 0$) defines a metric and real infinitesimal deformation isomorphic to $\mathfrak{osp}(1,2)\bigoplus \C_{1|0}$ via the change of basis:
$$e'_1=-e_2,\ e'_2=-\frac{2}{a_2t} e_3,\ e'_3=\frac{2}{a_2t}e_4+2e_{1},\ e'_4=e_4,\ e'_5=\frac{1}{\sqrt[]{a_2t}}e_5,\ e'_6=\frac{1}{\sqrt[]{a_2t}}e_6.$$
\end{proof}
\begin{conclusion}
In dimension $4|2$, except for $\mathfrak{g}_{4|2}^{2}(\frac{1}{2})$, every metric infinitesimal deformation is real.
We summarize the  metric deformation picture of $4|2$-dimensional indecomposable metric Lie superalgebras in the table below.
$$\begin{array}{llll}
  \mathrm{Algebra} & \mathrm{dim}\ \mathrm{H}_{\bar{0}}^{2}&\mathrm{Jump\ deformation} & \mathrm{Smooth\ deformation} \\
     \cline{1-4}
   \mathfrak{g}_{4|2}^{1} &  3 &  -  &\mathfrak{g}_{4|2}^{2}(0)=\mathfrak{b}\bigoplus \C_{0|2}\\
   \mathfrak{g}_{4|2}^{2}(\lambda), \lambda\neq 0,\pm \frac{1}{2}&2  &-& \mathfrak{g}_{4|2}^{2}(\lambda) \\
    \mathfrak{g}_{4|2}^{2}(\frac{1}{2})\cong\mathfrak{g}_{4|2}^{2}(-\frac{1}{2})& 4 & \mathfrak{g}_{4|2}^{3},\mathfrak{osp}(1,2)\bigoplus\C_{1|0} &\mathfrak{g}_{4|2}^{2}(\pm\frac{1}{2})\\
  \mathfrak{g}_{4|2}^{3}  & 2 & \mathfrak{osp}(1,2)\bigoplus\C_{1|0}   &-\\
 \cline{1-4}
 \end{array}$$
In particular, we get the picture of jump metric deformations as follows:
$$\xymatrix{
       \mathfrak{g}_{4|2}^{2}(\pm\frac{1}{2}) \ar[d]^{}          &             \\
  \mathfrak{g}_{4|2}^{3}  \ar[d]^{} &   \\
   \mathfrak{osp}(1,2)\bigoplus\C_{1|0}   &           }$$
where the down arrows show jump metric deformations.
\end{conclusion}

\subsection{$2|4$-dimensional metric Lie superalgebras.}There are the following indecomposable metric Lie superalgebras:
$\mathfrak{g}_{2|4}^{1}$, $\mathfrak{g}_{2|4}^{2}$, $\mathfrak{g}_{2|4}^{3}(\lambda)$ and $\mathfrak{g}_{2|4}^{4}$, with nontrivial brackets \cite{class}:
$$
\begin{tabular}{lrl}
  (1) & $\mathfrak{g}_{2|4}^{1}$ : & $[e_2,e_4]=e_3$,\ $[e_2,e_5]=-e_6$,\ $[e_4,e_5]=e_1$,\\
  (2) & $\mathfrak{g}_{2|4}^{2}$ : & $[e_2,e_4]=e_4$,\ $[e_2,e_5]=e_3$,\ $[e_2,e_6]=-e_6$,\ $[e_5,e_5]=[e_4,e_6]=e_1$, \\
  (3) & $\mathfrak{g}_{2|4}^{3}(\lambda)$  $(\lambda\neq 0)$ : & $[e_2,e_3]=e_3$,\ $[e_2,e_4]=\lambda e_4$,\ $[e_2,e_5]=-e_5$, \\
   &    &  $[e_2,e_6]=-\lambda e_6$,\ $[e_3,e_5]=e_1$,\ $[e_4,e_6]=\lambda e_1$, \\
  (4) & $\mathfrak{g}_{2|4}^{4}$ : &$[e_2,e_3]=e_3$,\ $[e_2,e_4]=e_3+e_4$,\ $[e_2,e_5]=-e_5-e_6$, \\
  &   &$[e_2,e_6]=-e_6$,\ $[e_3,e_5]=[e_4,e_5]=[e_4,e_6]=e_1$,\\
   (5) & $\mathfrak{g}_{2|4}^{5}$ : &$[e_2,e_3]=e_5$,\ $[e_2,e_4]=e_3$,\ $[e_2,e_5]=-e_6$, \\
  &   &$[e_3,e_3]=-e_1$,\ $[e_4,e_5]=e_1$.
\end{tabular}
$$
Here $\mathfrak{g}_{2|4}^{3}(\lambda)\mathop{\simeq}\limits^{i}\mathfrak{g}_{2|4}^{3}(-\lambda)\mathop{\simeq}\limits^{i}\mathfrak{g}_{2|4}^{3}(\lambda^{-1})$.
When  computing   metric deformations, we found a new class  metric Lie superalgebra $\mathfrak{g}_{2|4}^{5}$ (which is missing from the classification in \cite{class}).
Two invariant scalar products on this algebra can be  given by
$$\left(
  \begin{array}{cccccc}
    0 & 1 &  &  &  &  \\
    1 & 0 &  &  &  &  \\
     &  & 0 & 0 & 1 & 0 \\
     &  & 0 & 0 & 0 & 1 \\
     &  & -1 & 0 & 0 & 0 \\
     &  & 0 & -1 & 0 & 0 \\
  \end{array}
\right),
\qquad
\left(
  \begin{array}{cccccc}
    0 & 1 &  &  &  &  \\
    1 & 1 &  &  &  &  \\
     &  & 0 & 0 & 1 & 0 \\
     &  & 0 & 0 & 0 & 1 \\
     &  & -1 & 0 & 0 & 0 \\
     &  & 0 & -1 & 0 & 0 \\
  \end{array}
\right).$$
In addition, the algebra $\mathfrak{g}_{2|4}^{4}$ is also referred to $\mathfrak{g}_{6,7}^{s}$ in \cite{class}, although there is a mistake of a bracket in that paper.
\begin{thm}
The algebra  $\mathfrak{g}_{2|4}^{1}$ has jump metric deformations to $\mathfrak{g}_{2|2}^{2}\bigoplus \C_{0|2}$, $\mathfrak{g}_{2|4}^{2}$, $\mathfrak{g}_{2|4}^{3}(1)$, $\mathfrak{g}_{2|4}^{3}(\sqrt{-1})$, $\mathfrak{g}_{2|4}^{4}$, $\mathfrak{g}_{2|4}^{5}$ and it has smooth metric deformations to the family $\mathfrak{g}_{2|4}^{3}(\lambda)$ around $\lambda=0,1,\sqrt{-1}$.
\end{thm}
\begin{proof}
The even part of the  2-cohomology space is  9-dimensional, spanned by the  representative  even cocycles:
\begin{eqnarray*}
&&f_1=e^{2,3}_3,\ f_2=e^{2,3}_6,\ f_3=e^{2,6}_3,\ f_4=e^{2,6}_6,\ f_5=e^{1,2}_1+e^{4,6}_1,\ f_6=e^{2,3}_5-e^{3,3}_1, \\
&&f_7=e^{2,4}_4+e^{4,6}_1,\ f_8=e^{2,6}_4+e^{6,6}_1,\ f_9=e^{2,3}_4-e^{2,6}_5+e^{3,6}_1.
\end{eqnarray*}
If the cocycle $\sum_{i=1}^{9}a_if_i$ defines a metric infinitesimal deformation, then we have
$a_1=a_2=a_3=a_5=0, a_7=-a_4$ or $a_5=a_7=0, a_4=-a_1,a_6=a_2x,a_8=a_3x,a_9=a_1x$ for some $x\neq 0$.
From the brackets
$$[a_4(f_4-f_7)+a_6f_6+a_8f_8+a_9f_9,a_4(f_4-f_7)+a_6f_6+a_8f_8+a_9f_9]=0, $$
$$[a_1(f_1-f_4+xf_9)+a_2(f_2+xf_6)+a_3(f_3+xf_8),a_1(f_1-f_4+xf_9)+a_2(f_2+xf_6)+a_3(f_3+xf_8)]=0,
$$
all metric infinitesimal deformations are real infinitesimal deformations. The proof is complete by a similar discussion in the proof of Theorem \ref{dis}.
\end{proof}

\begin{thm}
The algebra $\mathfrak{g}_{2|4}^{2}$ has only one metric deformation, which is a smooth metric deformation to the family $\mathfrak{g}_{2|4}^{3}(\lambda)$ around $\lambda=0$.
\end{thm}
\begin{proof}
The  even part of the  2-cohomology space is  3-dimensional, spanned by the  representative  even cocycles:
$$f_1=e^{2,3}_3,\ f_2=e^{2,3}_5-e^{3,3}_1,\ f_3=e^{1,2}_1-e^{2,6}_6-\frac{1}{2}e^{3,5}_1.$$
If the cocycle $\sum_{i=1}^{3}a_if_i$ defines a metric infinitesimal deformation, then we have
$a_1=a_3=0$. From $[f_2,f_2]=0$, the cocycle $a_2f_2$ ($a_2\neq 0$) defines a metric and real infinitesimal deformation isomorphic to $\mathfrak{g}_{2|4}^{3}\left(\frac{1}{\sqrt{a_2t}}\right)\mathop{\simeq}\limits^{i} \mathfrak{g}_{2|4}^{3}(\sqrt{a_2t})$ via the change of basis:
$$e'_1=-2a_2te_1,\ e'_2=\frac{1}{\sqrt{a_2t}}e_2,\ e'_3=e_3+\sqrt{a_2t}e_5,\ e'_4=-2\sqrt[]{a_2t}e_4,\ e'_5=e_3-\sqrt{a_2t}e_5,\ e'_6=e_6.$$
\end{proof}
\begin{thm}
Metric deformations of the family $\mathfrak{g}_{2|4}^{3}(\lambda)$ $(\lambda\neq0)$ are the following:

(1) The generic element for $\lambda\neq \pm 1$   has only one metric deformation, which is a smooth deformation to the family $\mathfrak{g}_{2|4}^{3}(\lambda)$ around itself.

(2) The special element $\mathfrak{g}_{2|4}^{3}(1)\mathop{\simeq}\limits^{i}\mathfrak{g}_{2|4}^{3}(-1)$
has a smooth metric deformation to the family $\mathfrak{g}_{2|4}^{3}(\lambda)$ around itself and it has a jump metric deformation to $\mathfrak{g}_{2|4}^{4}$.
\end{thm}
\begin{proof}
Note that the case $\lambda=0$ is generally excluded from this family, because $\mathfrak{g}_{2|4}^{3}(0)=\mathfrak{g}_{2|2}^{2}\bigoplus\C_{0|2}$. Beside that, in the cases $\lambda=\pm1$, the cohomology and deformation patterns are not generic.

(1)
The even part of the  2-cohomology space is  2-dimensional, spanned by the  representative  even cocycles:
$$
f_1=e^{2,3}_3-e^{2,5}_5,\ f_2=e^{1,2}_1-e^{2,5}_5-e^{2,6}_6.
$$
If the cocycle $\sum_{i=1}^{2}a_if_i$ defines a metric infinitesimal deformation, then we have
$a_2=0$. From $[f_1,f_1]=0$, the cocycle $a_1f_1$ ($a_1\neq 0$) defines a metric and real infinitesimal deformation isomorphic to $\mathfrak{g}_{2|4}^{3}\left(\frac{\lambda}{1+a_1t}\right)$ via the change of basis:
$$e'_{1}=(1+a_1t)e_1,\ e'_2= \frac{1}{1+a_1t} e_2,\ e'_3=(1+a_1t)e_3,\ e'_4=e_4,\ e'_5=e_5,\ e'_6=e_6.$$

(2)
Because of the i-isomorphism, it is sufficient to consider $\mathfrak{g}_{2|4}^{3}(1)$. The even part of the  2-cohomology space is  4-dimensional, spanned by the  representative  even cocycles:
$$
f_1=e^{2,3}_3-e^{2,5}_5,\ f_2=e^{1,2}_1-e^{2,5}_5-e^{2,6}_6,\ f_3=e^{2,4}_3-e^{2,5}_6,\ f_4=e^{2,3}_4-e^{2,6}_5.
$$
If the cocycle $\sum_{i=1}^{4}a_if_i$ defines a metric infinitesimal deformation, then we have that
$a_2=a_3a_4=0$. We determine nontrivial metric deformation in the following cases:

\emph{Case 1}: $a_1\neq 0$, $a_3=0$. The cocycle $a_1f_1+a_4f_4$ ($a_1\neq 0$) defines a metric and real infinitesimal deformation isomorphic to $\mathfrak{g}_{2|4}^{3}(1+a_1t)$ via the change of basis:
$$e'_1=e_1,\ e'_2=e_2,\ e'_3=e_4,\ e'_4=e_3+\frac{a_4}{a_1}e_4,\ e'_5=e_6-\frac{a_4}{a_1}e_5,\ e'_6=(1+a_1t)e_5.$$

\emph{Case 2}: $a_1a_3\neq 0$, $a_4=0$. The cocycle $a_1f_1+a_3f_3$ ($a_1a_3\neq 0$) defines a metric and real infinitesimal deformation isomorphic to $\mathfrak{g}_{2|4}^{3}(1+a_1t)$ via the change of basis:
$$e'_1=e_1,\ e'_2=e_2,\ e'_3=e_4-\frac{a_3}{a_1}e_3,\ e'_4=e_3,\ e'_5=e_6,\ e'_6=(1+a_1t)(e_5+\frac{a_3}{a_1}e_6).$$

\emph{Case 3}: $a_1=a_3=0$, $a_4\neq0$. The cocycle $a_4f_4$ ($a_4\neq 0$) defines a metric and real infinitesimal deformation isomorphic to $\mathfrak{g}_{2|4}^{4}$ via the change of basis:
$$e'_1=a_4te_1,\ e'_2=e_2,\ e'_3=a_4te_4,\ e'_4=e_3,\ e'_5=e_6+a_4te_5,\ e'_6=a_4te_5.$$

\emph{Case 4}: $a_1=a_4=0$, $a_3\neq0$. The cocycle $a_3f_3$ ($a_3\neq 0$) defines a metric and real infinitesimal deformation isomorphic to $\mathfrak{g}_{2|4}^{4}$ via the change of basis:
$$e'_1=a_3te_1,\ e'_2=e_2,\ e'_3=a_3te_3,\ e'_4=e_4,\ e'_5=e_5+a_3te_6,\ e'_6=a_3te_6.$$
\end{proof}
\begin{thm}
The algebra $\mathfrak{g}_{2|4}^{4}$ has only one metric deformation, which is a smooth deformation to the family $\mathfrak{g}_{2|4}^{3}(\lambda)$ around $\lambda=1$.
\end{thm}
\begin{proof}
The  even part of the  2-cohomology space is  2-dimensional, spanned by the  representative  even cocycles:
$$f_1=e^{2,3}_4-e^{2,6}_5+e^{3,6}_1,\ f_2=e^{1,2}_1-2e^{2,6}_6+e^{3,6}_1+e^{4,6}_1.$$
If the cocycle $\sum_{i=1}^{2}a_if_i$ defines a metric infinitesimal deformation, then we have
$a_2=0$. From $[f_1,f_1]=0$, the cocycle $a_1f_1$ ($a_1\neq 0$) defines a metric and real infinitesimal deformation isomorphic to $\mathfrak{g}_{2|4}^{3}\left(\frac{1-\sqrt{a_1t}}{1+\sqrt{a_1t}}\right)$ via the change of basis:
\begin{eqnarray*}
&&e'_{1}=-2\sqrt{a_1t}(1+\sqrt{a_1t})e_1,\ e'_2= \frac{1}{\sqrt{a_1t}+1}e_2,\ e'_3=e_3+\sqrt{a_1t}e_4, \\
&& e'_4=-e_3+\sqrt{a_1t}e_4,\ e'_5=-e_6-\sqrt{a_1t}e_5,\ e'_6=-e_6+\sqrt{a_1t}e_{5}
\end{eqnarray*}
\end{proof}

\begin{thm}
The algebra $\mathfrak{g}_{2|4}^{5}$ has jump metric deformations to  $\mathfrak{g}_{2|4}^{2}$ and $\mathfrak{g}_{2|4}^{3}(\sqrt{-1})$ and it has a smooth metric deformation to the family $\mathfrak{g}_{2|4}^{3}(\lambda)$ around $\lambda=\sqrt{-1}$.
\end{thm}
\begin{proof}
The  even part of the  2-cohomology space is  4-dimensional, spanned by the  representative  even cocycles:
$$f_1=e^{1,2}_1+\frac{1}{2}e^{3,5}_1+\frac{3}{2}e^{4,6}_1,\ f_2=e^{2,4}_4+e^{4,6}_1,\ f_3=e^{2,5}_3+e^{5,5}_1,\ f_4=e^{2,6}_4+e^{6,6}_1.$$
If the cocycle $\sum_{i=1}^{4}a_if_i$ defines a metric infinitesimal deformation, then we have that
$a_1=a_2=0$. We determine nontrivial metric deformation in the following cases:

\emph{Case 1}: $a_4\neq 0$. Set $\alpha, \beta\in \C\backslash\{0\}$, such that
$a_4t\alpha^{4}-a_3t\alpha^{2}+1=0$ and $\beta=\sqrt{\alpha^{2}a_3t-1}$.
The cocycle $a_3f_3+a_4f_4$ ($a_4\neq 0$) defines a metric and real infinitesimal deformation isomorphic to $\mathfrak{g}_{2|4}^{3}(\beta)$ via the change of basis:
\begin{eqnarray*}
&&e'_1=2(a_3t\alpha^{2}-2)e_1,\ e'_2=\alpha e_2,\ e'_3=e_3-a_4t\alpha^{3}e_4+\alpha e_5-\alpha^{2}e_6,\\
&&e'_4=-\beta^{3}e_3+a_4t\alpha^{3}e_4-\alpha\beta^{2}e_5+\alpha^{2}\beta e_6,\ e'_5=e_3+\alpha^{3}a_4te_4-\alpha e_5-\alpha^{2}e_6,\\ &&e'_6=e_3+\frac{\alpha^{3}}{\beta^{3}}a_4te_4-\frac{\alpha}{\beta}e_5-\frac{\alpha^{2}}{\beta^{2}}e_6.
\end{eqnarray*}
In particular, if $a_3=0$, we get a jump metric deformation to $\mathfrak{g}_{2|4}^{3}(\sqrt{-1})$. If $a_3\neq 0$, we get a smooth metric deformation to the family $\mathfrak{g}_{2|4}^{3}(\lambda)$ around $\lambda=\sqrt{-1}$.

\emph{Case 2}: $a_4=0$, $a_3\neq 0$. The cocycle $a_3f_3$ ($a_3\neq 0$) defines a metric and real infinitesimal deformation isomorphic to $\mathfrak{g}_{2|4}^{2}$ via the change of basis:
 \begin{eqnarray*}
 &&e'_1=-2e_1,\ e'_2=\frac{1}{\sqrt{a_3t}}e_2,\ e'_3=-\frac{\sqrt{2}}{a_3t}e_{6},\ e'_4=e_3+\frac{1}{\sqrt{a_3t}}e_5-\frac{1}{a_3t}e_{6}, \\
&& e'_5=\sqrt{\frac{2}{a_3t}}e_5-\sqrt{2a_3t}e_{4},\
e'_6=e_{3}-\frac{1}{\sqrt{a_3t}}e_5-\frac{1}{a_3t}e_{6}.
 \end{eqnarray*}
\end{proof}

\begin{conclusion}
In dimension $2|4$, every metric infinitesimal deformation is real.
We summarize the  metric deformation picture of $2|4$-dimensional indecomposable metric Lie superalgebras in the table below.
$$\begin{array}{llll}
   \mathrm{Algebra} & \mathrm{dim}\ \mathrm{H}_{\bar{0}}^{2}&\mathrm{Jump\ deformation} & \mathrm{Smooth\ deformation} \\
     \cline{1-4}
\mathfrak{g}_{2|4}^{1}&9 &\mathfrak{g}_{2|4}^{3}(0),\ \mathfrak{g}_{2|4}^{2},\ \mathfrak{g}_{2|4}^{3}(\pm1),&\mathfrak{g}_{2|4}^{3}(0),\ \mathfrak{g}_{2|4}^{3}(\pm1),\ \mathfrak{g}_{2|4}^{3}(\sqrt{-1})\\ &&\mathfrak{g}_{2|4}^{3}(\sqrt{-1}),\ \mathfrak{g}_{2|4}^{4},\ \mathfrak{g}_{2|4}^{5}&\\
 \mathfrak{g}_{2|4}^{2} & 3  &  - &\mathfrak{g}_{2|4}^{3}(0)=\mathfrak{g}_{2|2}^{2}\bigoplus \C_{0|2}\\
\mathfrak{g}_{2|4}^{3}(\lambda),\lambda\neq 0,\pm 1& 2&-&\mathfrak{g}_{2|4}^{3}(\lambda)\\
\mathfrak{g}_{2|4}^{3}(1)\mathop{\simeq}\limits^{i}\mathfrak{g}_{2|4}^{3}(-1)& 4& \mathfrak{g}_{2|4}^{4}&\mathfrak{g}_{2|4}^{3}(\pm1)\\
 \mathfrak{g}_{2|4}^{4}&  2 &-    &\mathfrak{g}_{2|4}^{3}(\pm1)\\
 \mathfrak{g}_{2|4}^{5}& 4  & \mathfrak{g}_{2|4}^{2},\ \mathfrak{g}_{2|4}^{3}(\sqrt{-1}) &\mathfrak{g}_{2|4}^{3}(\sqrt{-1})\\
 \cline{1-4}
 \end{array}$$
In particular, we get the picture of jump metric deformations as follows:
$$
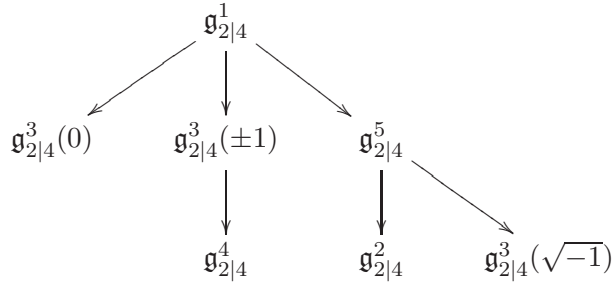
\xymatrix{
                &        \mathfrak{g}_{2|4}^{1}\ar[dr]\ar[dl] \ar[d]&  &   \\
  \mathfrak{g}_{2|4}^{3}(0)  & \mathfrak{g}_{2|4}^{3}(\pm1)\ar[d]  &  \mathfrak{g}_{2|4}^{5}\ar[d]\ar[dr]&  \\
 & \mathfrak{g}_{2|4}^{4}& \mathfrak{g}_{2|4}^{2} & \mathfrak{g}_{2|4}^{3}(\sqrt{-1}) }
$$
where the down arrows show jump metric deformations.
 \end{conclusion}

\small\noindent \textbf{Acknowledgment}\\
The author was supported by the CSC (No.201906170132) and the NSF of China (11771176). He  would like to thank Prof. Alice Fialowski for useful discussions.

\end{document}